\documentclass[10pt,reqno]{amsart}
\usepackage{graphicx, amssymb, color,pdfsync}
\usepackage[all,cmtip]{xy}
\numberwithin{equation}{section}
\usepackage{mathrsfs}
\usepackage{amsmath}

\usepackage[utf8]{inputenc}
\usepackage[T1]{fontenc}
\usepackage{mathtools}

\usepackage{tikz}
\usetikzlibrary{matrix,arrows}

\begin{document}
\newcommand{\s}{\vspace{0.2cm}}
\newcommand*{\transp}[2][-3mu]{\ensuremath{\mskip1mu\prescript{\smash{\mathrm t\mkern#1}}{}{\mathstrut#2}}}

\newtheorem{theo}{Theorem}
\newtheorem*{theo*}{Theorem}
\newtheorem{prop}{Proposition}
\newtheorem{coro}{Corollary}
\newtheorem{lemm}{Lemma}
\newtheorem{example}{Example}
\theoremstyle{remark}
\newtheorem{rema}{\bf Remark}
\newtheorem{claim}{\bf Claim}
\newtheorem{defi}{\bf Definition}
\newtheorem{conj}{Conjecture}

\title[The singular sublocus $\mathcal{A}_g^{\mathbb{Z}_2}$ is connected]{On the connectedness of the singular locus\\ of the moduli space of principally \\polarized abelian varieties}
\date{}

\author{Sebasti\'an Reyes-Carocca and Rub\'i E. Rodr\'iguez}
\address{Departamento de Matem\'atica y Estad\'istica, Universidad de La Frontera, Avenida Francisco Salazar 01145, Casilla 54-D, Temuco, Chile.}
\email{sebastian.reyes@ufrontera.cl, rubi.rodriguez@ufrontera.cl}

\thanks{Partially supported by Postdoctoral Fondecyt Grant 3160002, Fondecyt Grant 1141099 and Anillo ACT 1415 PIA-CONICYT}

\begin{abstract} Let $\mathcal{A}_g$ denote the moduli space of principally polarized abelian varieties of dimension $g \ge 3.$ In this paper we prove the connectedness of the singular sublocus of $\mathcal{A}_g$ consisting of those abelian varieties which possess an involution different from $-id$.
\end{abstract}
\maketitle

\section{Introduction and statement of the result}

Let $\mathcal{A}_g$ denote the moduli space of principally polarized abelian varieties of dimension $g \ge 3.$ It is a known fact that $\mathcal{A}_g$ is endowed with a structure of complex analytic space, and that its singular locus agrees with the set of points representing principally polarized abelian varieties admitting automorphisms different from $\pm id.$

\s

The problem of determining, if they exist, those $g$ for which the singular locus of $\mathcal{A}_g$ is connected is a natural problem to deal with; however, according to the knowledge of the authors, it still remains as an open question.

A contribution on this respect has been a decomposition of the singular locus of $\mathcal{A}_g$ in terms of irreducible components; this work is due to Gonz\'alez-Aguilera, Mu\~{n}oz-Porras and  Garc\'ia Zamora in \cite{VG}.

\s

The aforementioned connectedness problem is very much in contrast with the corresponding situation for the moduli space $\mathscr{M}_g$ of algebraic curves of genus $g.$ Indeed, by considering the equisymmetric stratification of the moduli space introduced by Broughton in \cite{b}, in a series of articles due to Bartolini, Costa,  Izquierdo and Porto, the connectedness of the singular locus of $\mathscr{M}_g$ has been completely determined according to the value of $g.$ More precisely, the singular locus of $\mathscr{M}_g$ is connected if and only if $g=4,7,13,17,19$ and $59.$ See, for example, \cite{m2}.

The analogous situation for the branch locus of the Schottky space has been considered in \cite{rumi}.

\s

We also mention that the connectedness of the singular sublocus $\mathscr{M}_g^{\mathbb{Z}_2}$ consisting of those points representing algebraic curves admitting an involution (see \cite[Theorem 6]{m1}) has played an important role --and is, indeed, the starting step-- in studying the connectedness of the singular locus of $\mathscr{M}_g.$

\s

Accordingly, we shall denote by $\mathcal{A}_g^{\mathbb{Z}_2}$ the singular sublocus of $\mathcal{A}_g$ consisting of those points representing principally polarized abelian varieties admitting an involution different from $-id.$

\s

The present article is mainly devoted to prove the natural analog of \cite[Theorem 6]{m1} in the context of principally polarized abelian varieties. More precisely:

\s

{\bf Main Theorem.}
Let $g \ge 3.$ The singular sublocus $\mathcal{A}_g^{\mathbb{Z}_2}$ of  $\mathcal{A}_g$ is connected.

\s

The authors were unable to find results related to the singular locus of the moduli space of principally polarized abelian surfaces (the case $g=2$ is not covered by Oort's theorem in [15]). By contrast, the case of dimension one (i.e. for elliptic curves) is trivial.

\s

This article is organized as follows.

\s

\begin{enumerate}
\item[(a)] In Section \ref{s2} we shall introduce the basic background; namely, complex tori, principally polarized abelian varieties, their moduli spaces and a classification of  symplectic involutions due to Reiner in \cite{reiner}.

\s

\item[(b)]  By using the aforementioned classification of involutions, in Section \ref{propos} we shall introduce a decomposition of  $\mathcal{A}_g^{\mathbb{Z}_2}$ and study intersections between some of its components.

\s

\item[(c)] The Main Theorem will be proved in Section \ref{s4ymedio}.

\s

\item[(d)] As a consequence of the results proved in \cite{VG}, a decomposition of $\mathcal{A}_g^{\mathbb{Z}_2}$ can be obtained in terms of irreducible subvarieties of it. In  Section \ref{s4} we shall study the relationship between this decomposition  and the one introduced in Section \ref{propos}.

\s

\item[(e)] In Section \ref{s5} we shall explore the relationship between the case of principally polarized abelian varieties and the one of algebraic curves.

\s

\item[(f)] In Section \ref{sejem} we shall illustrate our results for the case of principally polarized abelian threefolds.

\s

\item[(g)] Finally, in Section \ref{s6} we shall mention some remarks concerning future work on this respect.
\end{enumerate}

\section{Preliminaries} \label{s2}

\subsection{Complex tori}  We start this section by recalling known facts on complex tori and abelian varieties; we refer to the books \cite{bl} and \cite{debarre}.

A $g$-dimensional {\it complex torus} $X=V/\Lambda$ is the quotient of a $g$-dimensional complex vector space $V$ by a lattice $\Lambda$ of $V$. Each complex torus is an abelian group and a $g$-dimensional compact connected complex analytic manifold.

Complex tori can be described in a very concrete way, as follows. Choose bases \begin{equation} \label{bases}B_V=\{v_i\}_{i=1}^{g} \hspace{0.5 cm} \mbox{and} \hspace{0.5 cm} B_{\Lambda}=\{\lambda_j\}_{j=1}^{2g}\end{equation}of $V$ as a $\mathbb{C}$-vector space and of $\Lambda$ as a $\mathbb{Z}$-module, respectively. It follows that there are complex constants $$\{\pi_{ij}\}_{1 \le i \le g, 1\le j \le 2g} \hspace{0.5 cm} \mbox{such that} \hspace{0.5 cm}  \lambda_j=\Sigma_{i=1}^g \pi_{ij} v_i.$$

Then the matrix $$\Pi=(\pi_{ij}) \in \mbox{M}(g \times 2g, \mathbb{C})$$represents $X$ and is known as the {\it period matrix} for $X$ with respect to \eqref{bases}.

\subsection{Homomorphisms}  A {\it homomorphism} between two complex tori $X_1=V_1/\Lambda_1$ and $X_2=V_2/\Lambda_2$ is a holomorphic map which respects the group law. Each homomorphism $\varphi : X_1 \to X_2$ yields two maps:
\s

\begin{enumerate}
\item[(a)] the {\it analytic representation} $\rho_{a}(\varphi): V_1 \to V_2,$ which is the unique $\mathbb{C}$-linear map from $V_1$ to $V_2$ lifting $\varphi,$ and
\s

\item[(b)]  the {\it rational representation} $\rho_r(\varphi): \Lambda_1 \to \Lambda_2,$ which is the $\mathbb{Z}$-linear map from $\Lambda_1$ to $\Lambda_2$ given by the restriction of $\rho_a(\varphi)$ to $\Lambda_1.$
\end{enumerate}
\s

For $i=1,2,$ let $\Pi_i \in \mbox{M}(g_i \times 2g_i, \mathbb{C})$ denote the period matrix of $X_i$ with respect to chosen bases, where $g_i$ stands for the dimension of $X_i$. Then, with respect to these bases, $\rho_a(\varphi)$ and $\rho_r(\varphi)$ are given by matrices $$M \in \mbox{M}(g_2 \times g_1, \mathbb{C}) \,\,\, \mbox{ and } \,\,\,R \in \mbox{M}(2g_2 \times 2g_1, \mathbb{Z})$$respectively, that satisfy \begin{equation} \label{hom} M \Pi_1 = \Pi_2 R.\end{equation}

Conversely, given $M$ and $R$ satisfying \eqref{hom}, there is a corresponding homomorphism $X_1 \to X_2.$ Note that, if $\varphi$ is an {\it isomorphism} (i.e. a bijective homomorphism) then $M$ and $R$ are invertible.

\subsection{Abelian varieties} An {\it abelian variety} is by definition a complex torus which is also a complex projective algebraic variety. Each abelian variety $X=V/\Lambda$ admits a {\it polarization}; namely, a non-degenerate real alternating form $\Theta$ on $V$ such that for all $v,w \in V$$$\Theta(iv, iw)=\Theta(v,w) \hspace{0.5 cm} \mbox{and} \hspace{0.5 cm} \Theta(\Lambda \times \Lambda) \subset \mathbb{Z}.$$

We recall that the polarization determines an embedding of the complex torus in the projective space.

If the elementary divisors of $\Theta|_{\Lambda \times \Lambda}$ are $\{1, \stackrel{g}{\ldots}, 1\},$ where $g$ is the dimension of $X,$ then
the polarization $\Theta$ is called {\it principal} and the pair $(X,\Theta)$ is called a
{\it principally polarized abelian variety} (from now on we write {\it ppav} for short).

Let $(X=V/\Lambda, \Theta)$ be a ppav of dimension $g.$ Then there exists a basis for $\Lambda$ such that the matrix for $\Theta_{\Lambda \times \Lambda}$ with respect to it is given by \begin{equation}\label{simpl}
J = \left( \begin{smallmatrix}
0 & I_g \\
-I_g & 0
\end{smallmatrix} \right);
\end{equation}such a basis  is termed {\it symplectic.} Furthermore, there exist a basis for $V$ and a symplectic basis for $\Lambda$ with respect to which the period matrix for $X$ is $$\Pi=(I_g \, Z),$$where $Z$ belongs to the {\it Siegel upper-half space} $$\mathscr{H}_g=\{ Z \in \mbox{M}(g, \mathbb{C}) : Z = \transp{Z}, \, \mbox{Im}(Z) >0\}$$with $\transp{Z}$ denoting the transpose matrix of $Z.$

\subsection{The moduli space of ppavs} An {\it isomorphism} between two ppavs is an isomorphism of the underlying complex torus structures that preserves the polarizations. It is a known fact that the automorphism group of a ppav is finite.

For $i=1,2,$  let $(X_i,\Theta_i)$ be a ppav of dimension $g$ and let $\Pi_i=(I_{g} \, Z_i)$ be the period matrix of $X_i$. Then an isomorphism between $(X_1, \Theta_1)$ and $(X_2, \Theta_2)$ is given by invertible matrices $M \in \mbox{GL}(g, \mathbb{C})$ and $R \in \mbox{GL}(2g, \mathbb{Z})$ such that\begin{equation} \label{ig} M(I_{g} \, Z_1)=(I_{g} \, Z_2)R.\end{equation}

Since $R$ preserves the polarization \eqref{simpl}, it belongs to the symplectic group $$\mbox{Sp}(2g, \mathbb{Z})=\{ R \in  \mbox{M}(2g, \mathbb{Z}) :  \transp{R}J R=J  \}.$$

It follows from \eqref{ig} that the correspondence $\mbox{Sp}(2g, \mathbb{Z}) \times \mathscr{H}_g \to \mathscr{H}_g$
\begin{equation}  \label{cuadradat}     (R=  \left( \begin{smallmatrix}
A & B \\
C & D
\end{smallmatrix} \right) , Z ) \mapsto (A+ZC)^{-1}(B+ZD)
\end{equation}defines an action that identifies period matrices representing isomorphic ppavs.

Hence, the quotient $$\mathscr{H}_g \to \mathcal{A}_g:=\mathscr{H}_g/ \mbox{Sp}(2g, \mathbb{Z})$$is the {\it moduli space} of isomorphism classes of ppavs of dimension $g.$ See \cite{oort}.

We recall that an automorphism of a ppav represented by a period matrix $(I_g \, Z)$ with $Z \in \mathscr{H}_g$ is given by a  symplectic matrix $R$ as in \eqref{cuadradat} satisfying \begin{equation} \label{cheq}(A+ZC)^{-1}(B+ZD)=Z.\end{equation}In particular, if $R= \left( \begin{smallmatrix}
A & 0 \\
0 & D
\end{smallmatrix} \right)$ then \eqref{cheq} turns into $ZD=AZ,$ with $D=\transp{(A^{-1})}$.

\subsection{Reiner's classification of involutions} Let us denote by the symbol $+$ the direct sum of matrices, and by $I_n$ the  $n \times n$ identity matrix.

In \cite{reiner} Reiner succeeded in providing a classification of symplectic matrices $R \in \mbox{Sp}(2g, \mathbb{Z})$ such that $R^2=I_{2g}.$ Namely, each such $R$ is, up to symplectic conjugation, of the form $$W(x,y,z)+ \transp{W}(x,y,z)$$where $W(x,y,z)$ is the block matrix given by
\begin{displaymath}
W(x,y,z) =  J_1+ \stackrel{x}{\cdots} + J_1+ (-I_y)+I_z \,\, \,\mbox{ with } \,\,\, J_1 = \left( \begin{smallmatrix}
1 & 0 \\
1 & -1
\end{smallmatrix} \right),
\end{displaymath} for some non-negative integers $x,y,z$ such that $2x+y+z=g.$

\section{A decomposition of the singular sublocus $\mathcal{A}_g^{\mathbb{Z}_2}$}\label{propos}

Let $X$ be a ppav of dimension $g$ and let $j$ be an {\it involution} on it; namely, an order two automorphism of $X$ as ppav.

Let $x,y,z$ be non-negative integers such that $2x+y+z=g.$ We shall say that the involution $j$ is of {\it type} $(x,y,z)$ if its rational representation $\rho_r(j)$ is conjugate to
$$W(x,y,z)+ \transp{W}(x,y,z)$$in the symplectic group.

Note that there are two {\it extremal} types: namely, $(0,0,g)$ and $(0,g,0).$ The first case corresponds to $id,$ which is not an involution. By contrast, every ppav admits the involution $-id,$ which is of type $(0,g,0).$

\s

We recall that $\mathcal{A}_g^{\mathbb{Z}_2}$ denotes the singular sublocus of $\mathcal{A}_g$ consisting of those ppavs admitting an involution different from $-id$. Accordingly, we set \begin{equation} \label{co} \mathcal{A}_g^{\mathbb{Z}_2, (x,y,z)} =\{[X] \in \mathcal{A}_g^{\mathbb{Z}_2} : \exists \, j : X \to X \mbox{ involution of type } (x,y,z) \}\end{equation} for each $(x,y,z)$ in the collection of {\it admissible triples}$$\mathscr{P}_g=\{(x,y,z) \in (\mathbb{N} \cup \{0\})^3: 2x+y+z=g \}-\{(0,0,g), (0,g,0)\}.$$

It follows immediately that $\mathcal{A}_g^{\mathbb{Z}_2}$ admits the following decomposition:\begin{equation} \label{dec1}\mathcal{A}_g^{\mathbb{Z}_2}= \bigcup_{{(x,y,z) \in \mathscr{P}_g}} \mathcal{A}_g^{\mathbb{Z}_2, (x,y,z)} \end{equation}

We shall say that \eqref{co} is the {\it component} of type $(x,y,z)$ of $\mathcal{A}_g^{\mathbb{Z}_2}.$ Following \cite[Corollary 5.4]{surveyrubi}, its dimension  is $$\dim \mathcal{A}_g^{\mathbb{Z}_2, (x,y,z)}  = x^2+x+xy+xz+\tfrac{1}{2}(y^2+y+z^2+z).$$

\begin{prop} \label{p1}
Let $(x,y,z)$ be an admissible triple. Then $$\mathcal{A}_g^{\mathbb{Z}_2, (x,y,z)} = \mathcal{A}_g^{\mathbb{Z}_2, (x,z,y)}. $$
\end{prop}

\begin{proof}
Let $X$ be a ppav admitting an involution $j$ of type $(x,y,z).$ Possibly after an isomorphism, there exists a symplectic basis $B=\{\lambda_i\}$ for $\Lambda,$ where $X=V/\Lambda,$ with respect to which the rational representation of $j$ is given by  $$W(x,y,z)+\transp{W}(x,y,z).$$

With respect to $B,$ the rational representation of $(-id) \circ j$ is then given by \begin{displaymath}
 R =(-1)W(x,y,z)+(-1)\transp{W}(x,y,z),\end{displaymath}which is conjugate in $\mbox{Sp}(2g, \mathbb{Z})$ to $$W(x,z,y)+\transp{W}(x,z,y).$$
 Thus $(-id) \circ j$ is an involution of $X$ of type $(x,z,y),$ and the proof is done.
\end{proof}

This remaining part of this section is devoted to study intersections between the components in the decomposition \eqref{dec1}.

\subsection{Components $(x,y,z)$ with $x=0$.} An admissible tripe $(x,y,z)$ with $x=0$ has the form $$(0,y,g-y) \hspace{0.5 cm} \mbox{with}  \hspace{0.5 cm}  1 \le y \le g-1.$$

The following proposition asserts that all the components with $x=0$ have non-empty common intersection.

\begin{prop} \label{t1} The intersection \begin{equation} \label{conxcero}\bigcap_{1 \le  y \le  g-1} \mathcal{A}_g^{\mathbb{Z}_2, (0,y,g-y)}\end{equation}is not empty. Furthermore, it contains the $g$-dimensional family $\mathcal{F}_0$ consisting of those ppavs which are isomorphic as ppavs to the product of $g$ elliptic curves.
\end{prop}

\begin{proof} By Proposition \ref{p1} it is enough to consider the intersection \eqref{conxcero} with the index $y$ running in $\{1, \ldots, [\frac{g}{2}]\}.$

Let $X =V/\Lambda \in \mathcal{F}_0,$ say $$X \cong E_1 \times \cdots \times E_g$$where each $E_j$ is an elliptic curve.  Choose  $a_1, \ldots, a_g \in \mathscr{H}_1$ such that $$E_j \cong \langle 1 \rangle_{\mathbb{C}}/\langle 1, a_j \rangle_{\mathbb{Z}}.$$

It is not difficult to construct bases $B_V$ and $B_{\Lambda}=\{\lambda_i\}$ of $V$ and $\Lambda$ respectively, with respect to which the period matrix for $X$ is $(I_g \, Z),$ with $$Z=\mbox{diag}(a_1, \ldots, a_g).$$

For each $1 \le y \le [\tfrac{g}{2}],$ consider the automorphism $\Phi_y$ of $\Lambda$ given by
\begin{displaymath}
\lambda_l \mapsto \left\{ \begin{array}{ll}
 -\lambda_l & \textrm{if $l \in \{1, \ldots, y, g+1, \ldots, g+y\}$}\\
\,\,\,  \,\lambda_l & \textrm{if $l \in \{y+1, \ldots, g, g+y+1, \ldots, 2g\}.$}
\end{array} \right.
\end{displaymath}

It is straightforward to see that, with respect to $B_{\Lambda}$, the automorphism $\Phi_y$ can be represented by the matrix
$$W(0,y,g-y)+W(0,y,g-y).$$

Now, note that the equality $$W(0,y,g-y) \cdot Z=Z \cdot W(0,y,g-y)$$holds for each $1 \le y \le [\tfrac{g}{2}].$ Hence $\Phi_y$ induces an involution of $X$ of type $(0,y,g-y)$ for each $y,$ showing that $$X \in \bigcap_{1 \le y \le [\frac{g}{2}]} \mathcal{A}_g^{\mathbb{Z}_2, (0,y,g-y)}.$$Clearly, the collection $\mathcal{F}_0$ depends on $g$ moduli, and the proof is done.
\end{proof}

\subsection{Components of type $(x,y,z)$ with $x \ge 1$} An admissible triple $(x,y,z),$ with fixed $1 \le x \le [\tfrac{g}{2}]$, has the form $$(x,y,g-2x-y) \hspace{0.5 cm} \mbox{with}  \hspace{0.5 cm}  0 \le y \le g-2x.$$
The following result generalizes Proposition \ref{t1} to the case $x \ge 1.$

\begin{prop}\label{pp4} Let $1 \le x \le [\tfrac{g}{2}]$ fixed. The intersection \begin{equation} \label{ceni} \bigcap_{0 \le y \le g-2x} \mathcal{A}_g^{\mathbb{Z}_2, (x,y,g-2x-y)}\end{equation}is not empty. Furthermore, it contains the $(x^2-x+g)$-dimensional family $\mathcal{F}_x$ of ppavs which are isomorphic as ppavs to the product $$A_1  \times E_1 \times \cdots \times E_{g-2x}$$where $E_1, \ldots, E_{g-2x}$ are elliptic curves, and $A_1$ is a ppav of dimension $2x$ represented by the period matrix $(I_{2x} \, Z_1)$ where $Z_1 \in \mathscr{H}_{2x}$ is of the form
\begin{equation} \label{curacavi2}\left( \begin{smallmatrix}
X_1 & X_2 & \cdots  & X_x \\
\, & X_{x+1} & \cdots & X_{2x-1} \\
\, & \, & \ddots & \vdots \\
\, & \, & \, & X_{x(x+1)/2}
\end{smallmatrix} \right) \hspace{0.5 cm} \mbox{where} \hspace{0.5 cm} X_j=\left(  \begin{smallmatrix}
2a_{j} & a_{j} \\
 a_{j} & b_{j}
\end{smallmatrix} \right).\end{equation}
\end{prop}

\begin{proof} Note that if $g$ is even and $x=\tfrac{g}{2}$ then the  intersection \eqref{ceni} is simply the component of type $(\tfrac{g}{2},0,0).$ Hence, we may assume $g$ odd or $g$ even with $x \le \tfrac{g}{2}-1.$ Moreover, by Proposition \ref{p1}, it is enough to consider the intersection \eqref{ceni} with the index $y$ running in $\{0, \ldots, [\frac{g-2x}{2}]\}.$

\s

Let $X$ be a ppav admitting an involution $j$ of type $(x,0,g-2x).$ Then, possibly after an isomorphism, there exists a symplectic basis $B=\{\lambda_i\}$ for $\Lambda,$ where $X=V/\Lambda,$ with respect to which the rational representation of $j$ is given by the matrix $$W(x,0,g-2x)+\transp{W}(x,0,g-2x)$$

 Following \cite[Theorem 5.3]{surveyrubi}, the period matrix of $X$ with respect to $B$ is $(I_g \, Z)$ with $Z \in \mathscr{H}_g$ of the form  \begin{displaymath}
Z= \left( \begin{smallmatrix}
Z_1 & U \\
\transp{U} & Z_2
\end{smallmatrix} \right),
\end{displaymath}where $Z_1  \in \mbox{GL}(2x, \mathbb{C})$ is symmetric and $U \in \mbox{GL}(2x \times (g-2x), \mathbb{C})$ are of the form \begin{equation*} Z_1=\left( \begin{smallmatrix}
X_1 & X_2 & \cdots  & X_x \\
\, & X_{x+1} & \cdots & X_{2x-1} \\
\, & \, & \ddots & \vdots \\
\, & \, & \, & X_{x(x+1)/2}
\end{smallmatrix} \right) \hspace{0.5 cm} U= \left( \begin{smallmatrix}
U_1 & U_2 & \cdots  & U_{g-2x} \\
U_{g-2x+1} & \, & \cdots & U_{2(g-2x)} \\
\, & \, & \ddots & \vdots \\
\, & \, & \, & U_{x(g-2x)}
\end{smallmatrix} \right) \end{equation*}where $X_j=\left(  \begin{smallmatrix}
2a_{j} & a_{j} \\
 a_{j} & b_{j}
\end{smallmatrix} \right),$ $U_j=\left(  \begin{smallmatrix}
2d_{j}  \\
 d_{j}
\end{smallmatrix} \right),$ and $Z_2 \in \mbox{GL}(g-2x, \mathbb{C})$ is symmetric.

\s

We now restrict to the subcollection $\mathcal{F}_x$ of ppavs possessing an involution of type $(x,0,g-2x)$ with period matrix as before, but satisfying the following two additional conditions:
\begin{enumerate}
\item[(a)] $U=0,$ and
\item[(b)] if $Z_2=(b_{ij})_{i,j=2x+1}^{g}$ then $b_{ij} =0$ if and only if $i \neq j.$
\end{enumerate}In other words, we consider ppavs with period matrix $(I_g \, Z)$ with \begin{equation*} \label{llaves12}Z= Z_1  + \mbox{diag}(b_{2x+1,2x+1}, \ldots, b_{g,g}).\end{equation*}with $Z_1$ as in \eqref{curacavi2}.

For each $1 \le y \le [\tfrac{g-2x}{2}],$ we define an automorphism $\Phi_y$ of $\Lambda$ as follows:\begin{displaymath}
\Phi_{y}(\lambda_l) = \left\{ \begin{array}{ll}
 \lambda_l+\lambda_{l+1}  & \textrm{if $l=2 u - 1$ for some $1 \le u \le x$}\\
  -\lambda_l  & \textrm{if $l=2 u$ for some $1 \le u \le x$}\\
   -\lambda_l  & \textrm{if $2x+1 \le l \le 2x+y$ or  $2x+1+g \le l \le 2x+y+g$}\\
   \lambda_l  & \textrm{if $2x+y+1 \le l \le g$ or $2x+y+1+g \le l \le 2g$}\\
  \lambda_l  & \textrm{if $l=2 u -1 +g $ for some $1 \le u \le x$}\\
  \lambda_{l-1}-\lambda_l  & \textrm{if $l=2 u + g$ for some $1 \le u \le x$}
\end{array} \right.
\end{displaymath}

It is easy to see that for each $y$ the automorphism $\Phi_y$ is represented, with respect to $B,$ by the matrix $$W(x,y,g-2x-y)+\transp{W}(x,y,g-2x-y).$$

Now, the fact that $\left( \begin{smallmatrix}
1 & 0 \\
1 & -1
\end{smallmatrix} \right) \cdot X_j = X_j  \cdot \transp{\left( \begin{smallmatrix}
1 & 0 \\
1 & -1
\end{smallmatrix} \right)}$ implies that the equality$$W(x,y,g-2x-y) \cdot Z = Z \cdot \transp{ W}(x,y,g-2x-y)$$holds, for each $1 \le y \le [\tfrac{g-2x}{2}].$

Hence, we are in position to asserts that each ppav in $\mathcal{F}_x$ admits, in addition to an involution of type $(x,0,g-2x)$, involutions of type $(x,y,g-2x-y)$ for each $1 \le y \le [\tfrac{g-2x}{2}].$ This proves the first statement.

\s

Note that the family $\mathcal{F}_x$ depends on $$x(x+1)+(g-2x)=x^2-x+g$$moduli. Moreover, as the period matrices of the members of $\mathcal{F}_x$ have the form $(I_g \, Z)$ with $$Z= Z_1  + \mbox{diag}(b_{2x+1,2x+1}, \ldots, b_{g,g}),$$we can conclude that they can be decomposed as the product of a ppav of dimension $2x$ with period matrix $(I_{2x} \, Z_1),$ and $g-2x$ elliptic curves with period matrix $(1 \, b_{ii})$ for $2x+1 \le i \le g$.

The proof of the proposition is done.
\end{proof}

\subsection{Components with $x \ge 1$ and a component with $x=0$}

We now proceed to show that for each fixed $1 \le x \le [\tfrac{g}{2}]$, the family $\mathcal{F}_x$ intersects a component of the form $(0, y, g-y)$ for a suitable choice of $y.$

\begin{prop} \label{p4} Let $1 \le x \le [\tfrac{g}{2}]$ fixed. Assume $g$ odd, or $g$ even with $x \neq \tfrac{g}{2}.$ Then $$\mathcal{F}_x \subset \mathcal{A}_g^{\mathbb{Z}_2, (0,2x,g-2x)}.$$
\end{prop}

\begin{proof} Let $X=V/\Lambda$ be a member of $\mathcal{F}_x.$ According to the proof of Proposition \ref{pp4}, there is a symplectic basis $B=\{\lambda_i\}$ for $\Lambda$ with respect to which the period matrix of $X$ is $(I_g \, Z),$ with $Z \in \mathscr{H}_g$ of the form \begin{equation*} \label{llaves}Z= Z_1  + \mbox{diag}(b_{2x+1,2x+1}, \ldots, b_{g,g}).\end{equation*}with $Z_1$ as in \eqref{curacavi2}.

Define the automorphism $\Phi$ of $\Lambda,$ as follows:\begin{displaymath}
\Phi(\lambda_l) =  \left\{ \begin{array}{ll}
  -\lambda_l & \textrm{if $1 \le l \le 2x$ or $g+1 \le l \le g+2x$}\\
\,\,\,\, \lambda_l &  \textrm{if $2x+1 \le l \le g$ or $2x+1+g \le l \le 2g.$}
\end{array} \right.
\end{displaymath}

It is easy to see that $\Phi$ is represented, with respect to $B,$ by the matrix $$W(0,2x,g-2x)+W(0,2x,g-2x).$$

Now the result follows from the equality $$W(0,2x,g-2x) \cdot Z = Z \cdot W(0,2x,g-2x).$$
\end{proof}

It only remains to consider the case $g$ even and $x=\frac{g}{2}.$ Let $\widetilde{\mathcal{F}}_{g/2}$ denote the $g$-dimensional subfamily of $\mathcal{F}_{g/2}$ consisting of those ppavs with period matrix $Z_1$ as in \eqref{curacavi2}, but with $X_j=0$ if and only if $X_j$ is not in the diagonal.

 In other words,  \begin{equation} \label{mod}Z_1=X_1+X_{\frac{g+2}{2}}+ \cdots + X_{\frac{g(g+2)}{8}} \hspace{0.5 cm} \mbox{with} \hspace{0.5 cm} X_j=\left(  \begin{smallmatrix}
2a_{j} & a_{j} \\
 a_{j} & b_{j}
\end{smallmatrix} \right). \end{equation}

\begin{prop} \label{p5} If $g$ is even then$$\widetilde{\mathcal{F}}_{g/2} \subset \mathcal{A}_g^{\mathbb{Z}_2, (0, \frac{g}{2}, \frac{g}{2})}.$$
\end{prop}

\begin{proof}
Let $X=V/\Lambda$ be a member of $\widetilde{\mathcal{F}}_{g/2}.$ Let $B=\{\lambda_i\}$ be the symplectic basis for $\Lambda$ with respect to which the period matrix of $X$ is $(I_g \, Z)$ with $Z$ as in \eqref{mod}.

Define the automorphism $\Phi$ of $\Lambda$ as follows:\begin{displaymath}
\Phi(\lambda_l) =  \left\{ \begin{array}{ll}
  -\lambda_l & \textrm{if $1 \le l \le \tfrac{g}{2}$ or $g+1 \le l \le \tfrac{3}{2}g$}\\
\,\,\,\, \lambda_l &  \textrm{if $\tfrac{g}{2}+1 \le l \le g$ or $\tfrac{3}{2}g+1 \le l \le 2g$}
\end{array} \right.
\end{displaymath}

It is easy to see that $\Phi$ is represented, with respect to  $B,$ by the matrix $$W\left(0,\tfrac{g}{2},\tfrac{g}{2}\right)+W\left(0,\tfrac{g}{2},\tfrac{g}{2}\right).$$

Now, the result follows from the equality $$W \left(0,\tfrac{g}{2},\tfrac{g}{2}\right) \cdot Z = Z \cdot W \left(0,\tfrac{g}{2},\tfrac{g}{2}\right).$$
\end{proof}

We finish this section by giving a bound on the number of components in the decomposition \eqref{dec1}.

\begin{prop}
The number of components in the decomposition \eqref{dec1} is at most\begin{displaymath}
 \left\{ \begin{array}{ll}
\,\,\,\,\, \tfrac{g(g+6)}{8} & \textrm{if $g $ is even; }\\

\tfrac{(g+5)(g-1)}{8} & \textrm{if $g $ is odd.}
\end{array} \right.
\end{displaymath}
\end{prop}

\begin{proof}  Following the proof of Proposition 2.5 in \cite{VG}, the number of triples $(x,y,z)$ of non-negative integers $x,y,z$ such that $2x+y+z=g$ is  \begin{equation*} \label{doble}
 n= \left\{ \begin{array}{ll}
\,\,\,\,\,\tfrac{(g+2)^2}{4} & \textrm{if $g $ is even; }\\
\tfrac{(g+1)(g+3)}{4} & \textrm{if $g $ is odd.}
\end{array} \right.
\end{equation*}

According to Proposition \ref{p1}, we can ensure that:
\begin{enumerate}
\item[(a)] there are
\begin{displaymath}
s_0:= \left\{ \begin{array}{ll}
\tfrac{g-2}{2} & \textrm{if $g $ is even; }\\
\tfrac{g-1}{2} & \textrm{if $g $ is odd }
\end{array} \right.
\end{displaymath}pairs of triples of the form $(0,y,z)$ giving rise to $s_0$ components in the decomposition \eqref{dec1}, and
\item[(b)] for each fixed $1 \le x \le [\tfrac{g}{2}],$ there are\begin{displaymath}
s_x:= \left\{ \begin{array}{ll}
\,\,\,\tfrac{g-2x}{2} & \textrm{if $g $ is even; }\\
\tfrac{g-2x+1}{2} & \textrm{if $g $ is odd }
\end{array} \right.
\end{displaymath}pairs of triples of the form $(x,y,z)$ giving rise to $s_x$ components in the decomposition \eqref{dec1}.
\end{enumerate}

Thus after disregarding, in addition, the extremal triples $(0,g,0)$ and $(0,0,g),$ the number of components in the decomposition \eqref{dec1} is at most
\begin{displaymath}   n-2-\sum_{x=0}^{[\frac{g}{2}]}s_x=
 \left\{ \begin{array}{ll}
\,\,\,\,\, \tfrac{g(g+6)}{8} & \textrm{if $g $ is even; }\\

\tfrac{(g+5)(g-1)}{8} & \textrm{if $g $ is odd.}
\end{array} \right.
\end{displaymath}
\end{proof}

\section{Proof of the Main Theorem: $\mathcal{A}_g^{\mathbb{Z}_2}$ is connected} \label{s4ymedio}
Let $g \ge 3$ and let $x,y,z$ be non-negative integers such that $(x,y,z)$ belongs to the set of admissible triples $\mathscr{P}_g.$ 

\s

As in \cite{surveyrubi}, denote by $S(x,y,z)$  the set of matrices in $\mathscr{H}_g$ of the form $$
\left(
     \begin{array}{ccc}
        X & Y  & U  \\
        Y^t & Z_y  & 0_{y,z}  \\
        U^t &  0_{z,y}  & Z_z  \\
     \end{array}
   \right) ,
$$
where $X$ is a $2x \times 2x$ complex symmetric matrix
of the form
$$
X = \left(
      \begin{array}{cccc}
         X_1 & X_2  & \ldots  &  X_{x}  \\
          & X_{x+1}  & \ldots  & X_{2x-1}  \\
          &   & \ddots  &   \\
          &   &   & X_{\frac{x(x+1)}{2}}  \\
      \end{array}
    \right),  \text{ with } X_j = \left(
  \begin{array}{cc}
     2a_j & a_j  \\
     a_j & b_j  \\
  \end{array}
\right),
$$
$Y$ is a $2x \times y$ complex  matrix of the form
$$
Y = \left(
      \begin{array}{cccc}
         Y_1 & Y_2  & \ldots  &  Y_{y}  \\
         Y_{y+1} &  & \ldots  & Y_{2y}  \\
          &   & \vdots  &   \\
          &   &   & Y_{xy}  \\
      \end{array}
    \right),  \text{ with } Y_j = \left(
  \begin{array}{c}
     0   \\
     c_j   \\
  \end{array}
\right),
$$
$U$ is a $2x \times z$ complex matrix of the form
$$
U = \left(
      \begin{array}{cccc}
         U_1 & U_2  & \ldots  &  U_{z}  \\
         U_{z+1} &   & \ldots  & U_{2z}  \\
          &   & \vdots  &   \\
          &   &   & U_{xz}  \\
      \end{array}
    \right),  \text{ with } U_j = \left(
  \begin{array}{c}
     2d_j   \\
     d_j   \\
  \end{array}
\right),
$$
$Z_n$ is a complex symmetric $n \times n$ matrix, with $n=y$ first and then $n=z$,  and $0_{y,z}$
is the zero $y \times z$ matrix.

\s

According to \cite[Lemma 5.2]{surveyrubi} and \cite[Theorem 5.3]{surveyrubi}, the set $S(x,y,z)$ satisfies the following  property: for each ppav of dimension $g$ admitting an involution of type $(x,y,z)$, after isomorphism, there exists a symplectic basis with respect to which its period matrix is given by $(I_g \, Z)$ with $Z \in S(x,y,z)$. Hence, the canonical projection $\mathscr{H}_g \to \mathcal{A}_g$ restricts to a projection
$$
S(x,y,z) \to \mathcal{A}_g^{\mathbb{Z}_2, (x,y,z)} \, .
$$

We claim that each $S(x,y,z)$ is connected. Indeed, this follows  from the given explicit description of the elements of $S(x,y,z)$, from which it is clear that  $S(x,y,z)$ is closed under sums and under multiplication by positive real numbers; thus, it is convex and therefore it is connected, as claimed.

Now the proof of the Main Theorem follows directly from Propositions \ref{p1}, \ref{t1}, \ref{pp4}, \ref{p4} and \ref{p5}. \qed

\section{The relationship with the decomposition in \cite{VG}}  \label{s4}

Let $X=V/\Lambda$ be a ppav of dimension $g$ and let $n \ge 3$ be an integer. Consider the $n$-torsion points subgroup of $X$ $$X_n:=\{x \in X :  x+ \stackrel{n}{\cdots} +x=0 \} \cong \tfrac{1}{n} \Lambda / \Lambda \cong (\mathbb{Z}_n)^{2g}$$

Each automorphism $\varphi$ of $X$ restricts to a group automorphism $\varphi_n$ of $X_n;$ this is known as the {\it $n$-level structure} of $\varphi.$ Moreover, if the rational representation of $\varphi$ is given by a matrix $$R \in \mbox{Sp}(2g, \mathbb{Z}),$$then $\varphi_n$ can be represented by the matrix obtained after reducing the entries of $R$ modulo $n;$ it will be denoted by $$R_n \in \mbox{GL}(2g, \mathbb{Z}_n).$$

\s

Following \cite{VG}, let us denote by $\mathcal{A}_g(q, \varphi_n)$  the set of ppavs of dimension $g$ admitting an automorphism of prime order $q$ whose $n$-level structure is conjugate to $\varphi_n$ in  $\mbox{GL}(2g, \mathbb{Z}_n).$  The main result of \cite{VG} asserts that, for $g,n \ge 3$ and $q$ prime, the set $\mathcal{A}_g(q, \varphi_n)$ is an irreducible algebraic subvariety of $\mathcal{A}_g.$ Moreover, for fixed $n$ \begin{equation*} \label{dec2} \mbox{Sing}(\mathcal{A}_g)= \bigcup\mathcal{A}_g(q, \varphi_n),\end{equation*}where the union is taken over all ppavs $X$ of dimension $g,$ and over all non-conjugate $n$-level structures of automorphisms $\varphi$ of prime order $q$ of $X.$

In particular, for fixed $n$, the decomposition \begin{equation} \label{dec3} \mathcal{A}_g^{\mathbb{Z}_2}= \bigcup\mathcal{A}_g(2, j_n).\end{equation}is obtained, where the union is taken over all ppavs $X$ of dimension $g$ and over all non-conjugate $n$-level structures of involutions $j$ ($ \neq -id$) of $X.$

\s

The following result relates the decompositions \eqref{dec1} and \eqref{dec3} of $\mathcal{A}_g^{\mathbb{Z}_2}.$
\begin{prop}
There is a surjective correspondence \begin{equation*} \label{cih} \Xi : \{ \mathcal{A}_g^{\mathbb{Z}_2, (x,y,z)} : (x,y,z) \in \mathscr{P}_g \} \to \{ \mathcal{A}_g(2, j_n) : j \neq -id\, \mbox{ is an involution} \} \end{equation*} from the set of the components of the decomposition \eqref{dec1} onto the set of the components of the decomposition  \eqref{dec3}.
\end{prop}

\begin{proof}
Let $x,y,z$ be non-negative fixed integers such that $(x,y,z) \in \mathscr{P}_g.$ If $X$ is a ppav admitting an involution $j$ of type $(x,y,z),$ then we define $$\Xi(\mathcal{A}_g^{\mathbb{Z}_2, (x,y,z)} ):=\mathcal{A}_g(2, j_n)$$where $j_n: X_n \to X_n$ stands for the $n$-level structure of $j.$

We claim that $\Xi$ is well-defined; indeed, if $Y$ is another ppav admitting an involution $i$ of the same type $(x,y,z)$ then the rational representations of $j$ and $i$ are conjugate in the symplectic group. In other words, there exists a symplectic matrix $E \in \mbox{Sp}(2g, \mathbb{Z})$ such that $$ \rho_r(j) \cdot E = E \cdot  \rho_r(i),$$ and therefore the corresponding $n$-level structures $j_n$ and $i_n$ are conjugate by $E_n,$ which belongs to $\mbox{GL}(2g, \mathbb{Z}_n).$ Thus $$ \mathcal{A}_g(2, j_n) =  \mathcal{A}_g(2, i_n)$$ showing that $\Xi$ is well-defined, as desired.

The surjectivity is easy to see; in fact, given $\mathcal{A}_g(2, j_n)$ we have that $$\Xi(\mathcal{A}_g^{\mathbb{Z}_2, (x_0,y_0,z_0)} )=\mathcal{A}_g(2, j_n)$$where the tuple $(x_0,y_0,z_0)$ is chosen as the type of $j.$

The proof is done.
\end{proof}

The following corollary follows directly from the way in which $\Xi$ was constructed.

\begin{coro} \label{fono}
If $ \Xi(\mathcal{A}_g^{\mathbb{Z}_2, (x,y,z)})=\mathcal{A}_g(2, j_n)$ then $$\mathcal{A}_g^{\mathbb{Z}_2, (x,y,z)}  \subseteq \mathcal{A}_g(2, j_n).$$
\end{coro}

\section{The relationship with $\mathscr{M}_g^{\mathbb{Z}_2}$} \label{s5}

Let $\mathscr{M}_g$ denote the moduli space of smooth irreducible complex algebraic curves of genus $g \ge 2.$ It is a known fact that $\mathscr{M}_g$ is endowed with a structure of complex analytic space of dimension $3g-3.$ See, for example, \cite{Nag}.

\s

The connectedness of the singular locus of $\mathscr{M}_g$ has been extensively studied ever since, in 1888, Bolza  \cite{B} proved that the curve $$y^2=x^5-1$$represents an isolated singular point in $\mathscr{M}_2$; see also \cite{K}. Furthermore, by the results of Igusa \cite{Igusa}, this point is the unique singular point of $\mathscr{M}_2.$

It is known that the singular locus of $\mathscr{M}_3$ agrees with those points representing curves $S$ admitting non-trivial automorphisms, with the exception of those hyperelliptic ones whose reduced automorphism group (i.e. $\mbox{Aut}(S)/\langle \iota \rangle$ with $\iota$ denoting the hyperelliptic involution) is trivial; see \cite{Oort13}. Meanwhile, for $g \ge 4$ the singular locus of $\mathscr{M}_g$ agrees with the branch locus of the canonical projection $$T_g\to \mathscr{M}_g$$where $T_g$ denotes the Teichm\"{u}ller space of Riemann surfaces of genus $g.$ Thus, for $g \ge 4$ the singular locus of $\mathscr{M}_g$ corresponds to those points representing curves admitting non-trivial automorphisms; see \cite{Popp} and \cite{Rauch}, and also \cite{oort}.

\s

It was proved in \cite{b} that the singular locus of the moduli space of algebraic curves admits an equisymmetric stratification \begin{equation*} \mbox{Sing}(\mathscr{M}_g) = \bigcup \mathscr{M}_g^{\mathbb{Z}_q, \theta}\end{equation*}where the stratum $\mathscr{M}_g^{\mathbb{Z}_q, \theta}$ corresponds to the set of points representing curves of genus $g$ admitting the action of a cyclic group of prime order $q$ with fixed topological class $\theta$.

As mentioned in the introduction of this paper, nowadays, the connectedness or not of the singular locus of $\mathscr{M}_g$ is known, according to the value of $g.$

\s

Assume $g \ge 3$ and let $t:=[\tfrac{g+1}{2}].$ We proceed to list some known facts for the case $q=2.$
\begin{enumerate}

\s

\item[(a)] There are exactly $t+1$ topologically non-equivalent actions of $\mathbb{Z}_2$ on algebraic curves $C$ of genus $g.$ These actions, say $\theta_i,$ are parametrized
by the genus of the corresponding quotient; namely, the $$\mbox{action of type }\theta_i \iff \mbox{ signature }(i; 2, \stackrel{2g+2-4i}{\ldots}, 2)$$for $0 \le i \le t.$  We recall that the signature of the action is $(i; 2, \stackrel{2g+2-4i}{\ldots}, 2)$ means that the corresponding quotient Riemann surface has genus $i,$ and the associated two-fold branched regular covering map ramifies over exactly $2g+2-4i$ values.

Note that $\theta_0$ corresponds to the action given by the hyperelliptic involution.

\s

\item[(b)] In particular, the following decomposition of $\mathscr{M}_g^{\mathbb{Z}_2}$ is obtained:\begin{equation} \label{decm}
\mathscr{M}_g^{\mathbb{Z}_2}=\bigcup_{i=0}^{t} \mathscr{M}_g^{\mathbb{Z}_2, \theta_i}
\end{equation}Note that if $g=3,$ then the stratum with $i=0$ corresponds to those hyperelliptic curves with non-trivial reduced automorphism group.

\s

\item[(c)] Following \cite[Theorem 6]{m1}, there exists a $g$-dimensional family of algebraic curves of genus $g$ admitting simultaneously $\mathbb{Z}_2$-actions of type $\theta_i$ and of type $\theta_{t}$. In other words, $$\mathscr{M}_g^{\mathbb{Z}_2, \theta_i} \cap \mathscr{M}_g^{\mathbb{Z}_2, \theta_t} \neq \emptyset,$$for each $0 \le i \le t-1.$

\s

\item[(d)] Following \cite[Theorem 6.1]{surveyrubi}, if an algebraic curve $C$ admits a $\mathbb{Z}_2$-action of type $\theta_i,$ with $i \neq 0,$ then its Jacobian variety $JC$ (endowed with the canonical polarization) admits an involution of type:
\begin{displaymath} (x,y,z)=
 \left\{ \begin{array}{ll}
(\tfrac{g-1}{2},0,1) & \textrm{if $i=\tfrac{g+1}{2} $ and $g $ odd }\\
(i, g-2i,0) &  \textrm{otherwise}
\end{array} \right.
\end{displaymath}
\end{enumerate}

By (b) the stratum corresponding to the action $\theta_{t}$ is {\it special} with respect to the decomposition \eqref{decm}, in the sense that it has non-trivial intersection with all other strata. By (d) this {\it special} stratum is in correspondence with the component \begin{equation*}
 \left\{ \begin{array}{ll}
(\tfrac{g-1}{2},0,1) & \textrm{if $g $ odd }\\
(\tfrac{g}{2}, 0, 0) &  \textrm{if $g $ even }
\end{array} \right.
\end{equation*}
of the decomposition \eqref{dec1} of $\mathcal{A}_g^{\mathbb{Z}_2}.$

Consequently, by considering the Torelli map $\mathscr{M}_g \to \mathcal{A}_g$ (which associates to each curve $C$ its Jacobian $JC$) it follows immediately that $$\mathcal{A}_g^{\mathbb{Z}_2,(i,g-2i,0)} \cap \mathcal{A}_g^{\mathbb{Z}_2, (\frac{g-1}{2},0,1)} \neq \emptyset$$for each $1 \le i < \tfrac{g+1}{2}$ when $g$ is odd, and $$\mathcal{A}_g^{\mathbb{Z}_2,(i,g-2i,0)} \cap \mathcal{A}_g^{\mathbb{Z}_2, (\frac{g}{2}, 0, 0)} \neq \emptyset$$for each $1 \le i < \tfrac{g}{2}$ when $g$ is even (more precisely, the previous intersections contain a $g$-dimensional family of Jacobians).

\s

The following result assert that for $g$ odd each component of $\mathcal{A}_g^{\mathbb{Z}_2}$ is, in fact, {\it special} with respect to the decomposition \eqref{dec1}; namely, all components have non-empty common intersection.

\begin{prop} Let $g\ge 3$ be odd. There exists a $g$-dimensional family $\mathcal{F}$ of ppavs of dimension $g$ such that, if $X \in \mathcal{F},$ then:
\begin{enumerate}
\item[(a)] $X$ admits an involution of every type $(x,y,z) \in \mathscr{P}_g.$
\item[(b)]  $X$ is isomorphic as ppav to the product $$A_1 \times \cdots \times A_{\frac{g-1}{2}} \times E$$where $E$ is an elliptic curve, and each $A_j$ is a ppa surface represented by the period matrix $(I_2 \, Z_j)$ where $Z_j \in \mathscr{H}_2$ is the matrix of the form
\begin{equation*} \left(  \begin{smallmatrix}
2a_{j} & a_{j} \\
 a_{j} & b_{j}
\end{smallmatrix} \right).\end{equation*}
\item[(c)] the automorphism group of $X$ as ppav has order at least $\tfrac{(g+1)(g+3)}{4}.$
\end{enumerate}
\end{prop}

\begin{proof}
The proof follows the same arguments employed to prove the propositions of the third section.
\end{proof}

\s

{\bf Remarks.} \mbox{}
\begin{enumerate}
\item[(a)] In order to prove the Main Theorem of the paper, the previous proposition can replace Propositions \ref{p1}, \ref{t1}, \ref{pp4}, \ref{p4} and \ref{p5} for the case $g$ odd.
\item[(b)] As an application of \cite[Proposition 4.4]{anita}, it can be seen that the ppa surfaces $A_j$ of the previous proposition are isogenous to the product of two elliptic curves.
\end{enumerate}
\section{Principally polarized abelian threefolds} \label{sejem}

In this section we consider the decomposition \eqref{dec1} for ppa threefolds:\begin{equation*} \mathcal{A}_3^{\mathbb{Z}_2}= \bigcup_{{(x,y,z) \in \mathscr{P}_3}} \mathcal{A}_3^{\mathbb{Z}_2, (x,y,z)} \end{equation*}

Note that the admissible triples in this case are exactly four; namely: $$(0,1,2), (0,2,1), (1,0,1),(1,1,0).$$

By Proposition 1 we can ensure that the components associated to $(1,0,1)$ and to $(1,1,0)$ agree, and that the components associated to $(0,1,2)$ and to $(0,2,1)$ agree. So, in this case the singular sublocus $\mathcal{A}_3^{\mathbb{Z}_2}$ can be written as a union of only two components, each one of dimension 4: \begin{equation} \label{dd1}\mathcal{A}_3^{\mathbb{Z}_2}=\mathcal{A}_3^{\mathbb{Z}_2, (0,1,2)} \cup \mathcal{A}_3^{\mathbb{Z}_2, (1,0,1)}\end{equation}

Let $S$ be a ppa surface described by a period matrix of the form $\left(  \begin{smallmatrix}
2a & a \\
 a & b
\end{smallmatrix} \right)$ and let $E$ be an elliptic curve. Then, by Proposition \ref{p4}, each member of the three-dimensional family $\mathcal{F}_1$ of ppa threefolds of the form $S \times E$ admits simultaneously involutions of type $(0,1,2)$ and $(1,0,1),$ showing that  \begin{equation} \label{dd2}\mathcal{A}_3^{\mathbb{Z}_2, (0,1,2)} \cap \mathcal{A}_3^{\mathbb{Z}_2, (1,0,1)} \neq \emptyset.\end{equation}

 Following Proposition \ref{t1}, the component of type $(0,1,2)$ contains the family $\mathcal{F}_0$ of ppa threefolds which are isomorphic to the product of three elliptic curves. Note that the families $\mathcal{F}_0$ and $\mathcal{F}_1$ (and, more generally, $\mathcal{F}_x$ for each $x \ge 0$) do not contain Jacobians, and that $\mathscr{F}_0$ does not intersect $\mathscr{F}_1.$

\s

Let us assume that $n=p$ is an odd prime number, and that $j$ and $i$ are involutions of type $(0,1,2)$ and $(1,0,1)$ respectively. Then, in spite of the fact that the rational representations of $j$ and $i$ are not conjugate in $\mbox{Sp}(6, \mathbb{Z}),$  the equality
\begin{equation*} \left( \begin{smallmatrix}
M & 0 \\
0 & \transp{M}
\end{smallmatrix} \right) \cdot \left( \begin{smallmatrix}
W(1,0,1) & 0 \\
0 & \transp{W}(1,0,1)
\end{smallmatrix} \right) = \left( \begin{smallmatrix}
W(0,1,2) & 0 \\
0 & {W}(0,1,2)
\end{smallmatrix} \right) \cdot \left( \begin{smallmatrix}
M & 0 \\
0 & \transp{M}
\end{smallmatrix} \right)
\end{equation*}where $$M=\left( \begin{smallmatrix}
1 & -2 & 0 \\
c & 0 & 0 \\
0 & 0 & 1
\end{smallmatrix} \right) \,\,\, \mbox{ with } \,\,\, -2c \equiv 1 \, \mbox{mod} \, p,$$shows that their restrictions $j_p$ and $i_p$ are conjugate in $\mbox{GL}(6,\mathbb{Z}_p).$ In other words, $$\mathcal{A}_3(2, j_p)=\mathcal{A}_3(2, i_p)=\mathcal{A}_3^{\mathbb{Z}_2},$$and therefore $\mathcal{A}_3^{\mathbb{Z}_2}$ is irreducible.

\s

We recall that, as mentioned in Section \ref{s5}, the component of type $(1,0,1)$ contains Jacobians and, by contrast, the component of type $(0,1,2)$ does not. Thus, the fact that $\mathcal{A}_3^{\mathbb{Z}_2}$ is irreducible together with \eqref{dd1} and \eqref{dd2} allow us to state that $$\mathcal{A}_3^{\mathbb{Z}_2, (0,1,2)} \subset \mathcal{A}_3^{\mathbb{Z}_2, (1,0,1)}=\mathcal{A}_3^{\mathbb{Z}_2}.$$

In particular we obtain:

\begin{prop}
If a ppa threefold admits an involution different from $-id$ then it admits an involution of type $(1,0,1).$
\end{prop}

\section{Final Remarks} \label{s6}

We end this article by mentioning some remarks:

\s

\begin{enumerate}

\item[(a)] To generalize our results from involutions to higher order automorphisms seems a difficult task; the main difficulty is that there does not exist a symplectic classification as it does in the case of involutions. We refer to the master thesis \cite{torres} for an attempt to finding an analogous classification in the order three case.

\s

\item[(b)] In \cite{VGL} the automorphisms of polarized abelian threefolds were classified; among them, in particular, ppa threefolds. Then, following the arguments employed for the case of algebraic curves of genus three, it seems reasonable to ask if this classification together with our results can be used to decide if the singular locus of $\mathcal{A}_3$ is or is not connected.

\s

\item[(c)] The number of isolated points in the singular locus of $\mathscr{M}_g$ is classically known for each $g \ge 2;$ see \cite{K}. In an analogous way, one can ask whether the singular locus of $\mathcal{A}_g$ has isolated points (see \cite{VG2} for the case $2g+1$ prime).
Towards answering the aforesaid question, we have noticed the following simple --but apparently unknown-- fact: if there exists an isolated singular point $X$ in $\mathcal{A}_g$, then the order of its reduced automorphism group must be odd. In fact, if the reduced automorphism group of $X$ were of even order, then $X$ would belong to some component $(x,y,z)$ and, consequently, it would not be isolated.

\s

\item[(d)] The decomposition \eqref{dec3} depends heavily on the choice of the integer $n;$ this dependence being reflected, for instance, in the number of components (if $g=3$ and $n=6,$ then  \eqref{dec3} has two components, instead of one as in the case $n=p$ prime). It would be interesting to understand the extent to which the number of components in \eqref{dec3} is affected by the choice of $n;$ this problem was unconsidered in \cite{VG}.

\end{enumerate}


\begin{thebibliography}{9}
\bibitem{anita}
{\sc R. Auffarth, H. Lange and A. M. Rojas},  { \em A criterion for an abelian variety to be non-simple}, J. Pure Appl. Algebra {\bf 221} (2017), No. 8, 1906--1925.

\bibitem{m2}
{\sc G. Bartolini, A. F. Costa and M. Izquierdo},  { \em On the connectedness of branch loci of moduli spaces}. Ann. Acad. Sci. Fenn. Math. {\bf 38} (2013), no. 1, 245--258.

\bibitem{m1}
{\sc G. Bartolini and M. Izquierdo},
 { \em On the connectedness of the branch loci of the moduli spaces of Riemann surfaces of low genus,} Proc. Amer. Math. Soc. {\bf 140} (2012), no. 1, 35--45.

\bibitem{bl}
{\sc Ch. Birkenhake and H. Lange},
 { \em Complex Abelian Varieties,} $2^{nd}$ edition,
Grundl. Math. Wiss. {\bf 302}, Springer, 2004.

\bibitem{VGL}
{\sc Ch. Birkenhake, V. Gonz\'alez-Aguilera and H. Lange},
 { \em Automorphisms of 3-dimensional abelian varieties,} Complex geometry of groups (Olmu\'e, 1998), 25--47, Contemp. Math., {\bf 240}, Amer. Math. Soc., Providence, RI, 1999.



\bibitem{B}
{\sc O. Bolza},
 { \em On binary sexties with linear transformations between themselves,} Amer. J. Math. {\bf 10} (1888), 47--70.

\bibitem{b}
{\sc S. A. Broughton},
 { \em The equisymmetric stratification of the moduli space and the Krull dimension of mapping class groups,} Topology Appl. {\bf 37} (1990), no. 2, 101--113.

\bibitem{debarre}
{\sc O. Debarre,}
{\em Tores et vari\'et\'es ab\'eliennes complexes,}
Cours Sp\'ecialis\'es, {\bf 6}. Soci\'et\'e Math\'ematique de France, Paris; EDP Sciences, Les Ulis (1999).

\bibitem{VG}
{\sc V. Gonz\'alez-Aguilera, J. M. Mu\~{n}oz-Porras and A. G. Zamora},
 { \em On the irreducible components of the singular locus of $A_g$,} J. Algebra {\bf 240} (2001) no. 1, 230--250.

\bibitem{VG2}
{\sc V. Gonz\'alez-Aguilera, J. M. Mu\~{n}oz-Porras and A. G. Zamora},
 { \em On the 0-dimensional irreducible components of the singular locus of $A_g$,} Arch. Math. {\bf 84} (2005), 298--303.


\bibitem{rumi}
{\sc R. A. Hidalgo and M. Izquierdo},
 { \em On the connectivity of the branch locus of the Schottky space,} Ann. Acad. Sci. Fenn. Math.  {\bf 39} (2014), No. 2, 635--654.


\bibitem{Igusa}
{\sc J. I. Igusa},
 { \em Arithmetic variety of moduli for genus two,} Ann. Math. {\bf 72} (1960), 612--649.





\bibitem{K}
{\sc R. Kulkarni},
 { \em Isolated points in the branch locus of the moduli space of compact Riemann surfaces,} Ann. Acad. Sci. Fenn. Ser. A I Math. {\bf 16} (1991),  71--81.

\bibitem{Nag}
{\sc{S. Nag,}} {\em{The complex analytic theory of Teichm\"{u}ller spaces,}} Canadian Math. Soc. Series of Monographs and Advanced Texts, Wiley-Intersciences (1988).

\bibitem{oort}
{\sc F. Oort},
 { \em Singularities of coarse moduli schemes,} S\'em. Dubriel {\bf 16} (1976).

\bibitem{Oort13}
{\sc F. Oort},
 { \em Singularities of the moduli scheme for curves of genus three,} Proc. Kon. Ned. Akad. {\bf 78} (1975), 170--174.


\bibitem{Popp}
{\sc H. Popp},
 { \em The singularities of the moduli scheme of curves,} Journ. number theory {\bf 1} (1969), 90--107.

\bibitem{Rauch}
{\sc H. E.  Rauch},
 { \em The singularities of the modulus space,} Bull. Amer. Math. Soc. {\bf 68} (1962), 390--394.

\bibitem{reiner}
{\sc I. Reiner},
 { \em Automorphisms of the symplectic modular group,} Trans. Amer. Math. Soc. {\bf 80} (1955), 35--50.

\bibitem{surveyrubi}
{\sc R. E. Rodr\'iguez}, {\em Abelian varieties and group actions}, Riemann and Klein surfaces, automorphisms, symmetries and moduli spaces, 299--314, Contemp. Math., {\bf 629}, Amer. Math. Soc., Providence, RI, 2014.


\bibitem{torres}
{\sc S. Torres},
 { \em Complex abelian varieties,} Master thesis, Pontifica Universidad Cat\'olica de Chile, 2015. Available in: www.mat.uc.cl/archivos/dip/tesis-postgrado/magister-mat/tesis-sebastian-torres.pdf

\end{thebibliography}
\end{document}